\documentclass[11pt]{amsart}
\usepackage{amsthm}
\usepackage{amsmath,amstext,amsthm,amsfonts,amssymb,amsthm}
\usepackage{multicol}
\usepackage{latexsym}
\usepackage{color}
\usepackage{graphicx}
\usepackage{epsf}
\usepackage{epsfig}
\usepackage{epic}
\usepackage{graphics}
\usepackage{verbatim}
\usepackage{color}
\usepackage{hyperref}
\newtheorem{theorem}{Theorem}[section]
\newtheorem{lemma}[theorem]{Lemma}

\theoremstyle{definition}
\newtheorem{definition}[theorem]{Definition}
\newtheorem{proposition}[theorem]{Proposition}

\theoremstyle{remark}

\numberwithin{equation}{section}

\newcommand{\D}{\mathcal{D}}

\newcommand{\HQ}{\mathbb{H}}

\newcommand{\B}{\mathcal{B}}

\newcommand{\oqu}{\overline{q}}

\begin{document}
\title[Frame Multipliers]{Frame Multipliers for discrete frames on  Quaternionic Hilbert Spaces}
\author{M. Khokulan$^1$, K. Thirulogasanthar$^2$.}
\address{$^{1}$ Department of Mathematics and Statistics, University of Jaffna, Thirunelveli, Jaffna, Srilanka. }
\address{$^{2}$ Department of Computer Science and Software Engineering, Concordia University, 1455 de Maisonneuve Blvd. West, Montreal, Quebec, H3G 1M8, Canada.}
\email{mkhokulan@gmail.com, santhar@gmail.com.}
\date{\today}
\begin{abstract}
In this note, following the complex theory, we examine discrete controlled frames, discrete weighted frames and frame multipliers in a non-commutative setting, namely in a left quaternionic Hilbert space. In particular, we show that the controlled frames are equivalent to usual frames under certain conditions. We also study  connection between frame multipliers and weighted frames in the same Hilbert space.
\end{abstract}
\thanks{K. Thirulogasanthar would like to thank the, FRQNT, Fonds de la Recherche  Nature et  Technologies (Quebec, Canada) for partial financial support under the grant number 2017-CO-201915. Part of this work was done while he was visiting the University of Jaffna, Sri Lanka. He expresses his thanks for the hospitality.}
\maketitle
\pagestyle{myheadings}
\section{Introduction}
The notion of frames in Hilbert spaces was introduced by Duffin and Schaeffer in 1952 to address some very deep problems in non-harmonic Fourier series \cite{DU}. However the fundamental concept of frames was revived in 1980s by Daubechies, Grossmann and Meyer\cite{D1,D2}, who showed its significance in signal processing.
Frame is a spanning set of vectors which are generally overcomplete (redundant) in a Hilbert space. Therefore a typical frame contains more frame vectors than the dimension of the space and each vector in the space will have infinitely many representations with respect to the frame. This redundancy of frames is the key to their success in applications.\\

Nowadays frames have broad applications in Mathematics and Engineering in several areas including sampling theory\cite{Ald},  operator theory\cite{Han},  harmonic  analysis\cite{Lac},  wavelet theory \cite{D3},  wireless communication\cite{Sh1,Sh2},  data transmission with erasures \cite{Bod,GO1},  filter banks \cite{Bol},  signal processing \cite{Ben,GO2}, image processing \cite{Can},  geophysics \cite{Mar1} and quantum computing \cite{Eld}. Hilbert spaces can be defined over the fields $\mathbb{R}$, the set of all real numbers, $\mathbb{C}$, the set of all complex numbers, and $\HQ$, the set of all quaternions only \cite{Ad}. The fields $\mathbb{R}$ and $\mathbb{C}$ are associative and commutative but quaternions form a non-commutative associative algebra and this feature highly restricted mathematicians to work out a well-formed theory of functional analysis and harmonic analysis on quaternionic Hilbert spaces. The quaternionic frames have been developed in the mathematical point of view very recently \cite{KH, Kho}. The applications are yet to be identified and analyzed in terms of these frames.\\

Frame multiplier is an operator which was introduced by P.Balazs  for frames in Hilbert spaces \cite{Bal1}. In these multipliers the analysis coefficients are multiplied by a fixed sequence (called the symbol) before re-synthesis. Fundamental properties of this multiplier were investigated in \cite{Bal1}. Frame multipliers are interesting not only from a theoretical point of view but also for applications. For example, it is useful in psycho-acoustical modeling \cite{Bal2}, denoising \cite{Bal3},  computational auditory scene analysis \cite{Wan}, virtual acoustics \cite{Maj} and seismic data analysis \cite{Mar}.\\

The extensions of frames in Hilbert spaces include weighted and controlled frames and these extensions were introduced recently to improve the numerical efficiency of iterative algorithms for inverting the frame operator\cite{Bal4}. In this paper we investigate the connection between the weighted frames and frame multipliers for discrete frames in left quaternionic Hilbert spaces along the lines of the argument of P. Balazs et al. \cite{Bal4}.\\

This article is organized as follows: In section 2, we collect basic notations and some preliminary results about quaternions and frames as needed for the development of the results obtained in this article. In section 3, we present the concept of controlled frames in quaternionic Hilbert spaces and we will show that controlled frames are equivalent to usual frames under certain conditions. In section 4, we investigate the weighted frames and frame multipliers. We also investigate how the frame multipliers are connected with weighted frames in quaternionic Hilbert spaces.  Section 5 ends the article with a conclusion.

\section{Mathematical preliminaries}
\noindent
We recall few facts about quaternions, quaternionic Hilbert spaces and quaternionic functional analysis which may not be very familiar to the reader. 
\subsection{Quaternions}
Let $\mathbb{H}$ denote the field of quaternions. Its elements are of the form $q=x_0+x_1i+x_2j+x_3k,~$ where $x_0,x_1,x_2$ and $x_3$ are real numbers, and $i,j,k$ are imaginary units such that $i^2=j^2=k^2=-1$, $ij=-ji=k$, $jk=-kj=i$ and $ki=-ik=j$. The quaternionic conjugate of $q$ is defined to be $\overline{q} = x_0 - x_1i - x_2j - x_3k$. Quaternions do not commute in general. However $q$ and $\oqu$ commute, and quaternions commute with real numbers. $|q|^2=q\oqu=\oqu q$ and $\overline{qp}=\overline{p}~\oqu.$
\subsection{Left Quaternionic Hilbert Space}
Let $V_{\mathbb{H}}^{L}$ is a vector space under left multiplication by quaternionic scalars, where $\mathbb{H}$ stands for the quaternion algebra. For $f,g,h\in V_{\mathbb{H}}^{L}$ and $q\in \mathbb{H},$ the inner product
\begin{eqnarray*}
	\left\langle .|.\right\rangle:V_{\mathbb{H}}^{L}\times V_{\mathbb{H}}^{L}\longrightarrow \mathbb{H}
\end{eqnarray*}
satisfies the following properties:
\begin{enumerate}
	\item [(a)]$\overline{\left\langle f|g\right\rangle}=\left\langle g|f\right\rangle$
	\item [(b)]$\left\|f\right\|^{2}=\left\langle f|f\right\rangle>0~$ unless $~f=0,~$ a real norm
	\item [(c)]$\left\langle f|g+h\right\rangle=\left\langle f|g\right\rangle+\left\langle f|h\right\rangle$
	\item [(d)]$\left\langle qf|g\right\rangle=q\left\langle f|g\right\rangle$
	\item[(e)] $\left\langle f|qg\right\rangle=\left\langle f|g\right\rangle\overline{q}.$
\end{enumerate}
where $\overline{q}$ stands for the quaternionic conjugate.
We assume that the
space $V_{\mathbb{H}}^{L}$ is complete under the norm given above. Then,  together with $\langle\cdot\mid\cdot\rangle$ this defines a right quaternionic Hilbert space, which we shall assume to be separable. Quaternionic Hilbert spaces share most of the standard properties of complex Hilbert spaces. In particular, the Cauchy-Schwartz inequality holds on quaternionic Hilbert spaces as well as the Riesz representation theorem for their duals.\\
In the left quaternionic Hilbert space scalar multiplication of a vector is always from the left  to a vector. Therefore, the Dirac bra-ket notation can be adapted as
\begin{equation}\label{e31}
|q\phi\rangle=|\phi\rangle\overline{q},~~ \langle\phi|=|\phi\rangle^{\dagger},~~\langle q\phi|=q\langle \phi|   
\end{equation}
for a left quaternionic Hilbert space,with $|\phi\rangle $ denoting the vector $\phi$ and $\langle\phi| $ is its dual vector.\\  
Let $T$ be an operator on a left quaternionic Hilbert space $V_\mathbb{H}^L$ with domain $V_\mathbb{H}^L$.
The adjoint $T^{\dagger}$ of $T$ is defined as
\begin{equation}\label{Ad1}
\langle \psi \mid T\phi \rangle=\langle T^{\dagger} \psi \mid\phi \rangle;\quad\text{for all}~~~\phi ,\psi \in V_\mathbb{H}^L.
\end{equation}
An operator $T$ is said to be self-adjoint if $T=T^{\dagger}$.\\ 
Let $\D(T)$ denotes the domain of $T$. The operator $T$ is said to be left  linear if
\begin{eqnarray*}
	T(\phi+\psi)&=&T\phi+T\psi,\\
	T(q\phi)&=&q(T\phi)
\end{eqnarray*}
for all $\phi,\psi\in \mathcal{D}(T)$ and $q\in \mathbb{H}.$
The set of all left linear operators in $V_\mathbb{H}^L$  will be denoted by $\mathcal{L}(V_\mathbb{H}^L)$. For a given $T\in \mathcal{L}(V_\mathbb{H}^L)$, the range and the kernel will be
\begin{eqnarray*}
	\mbox{ran}(T)&=&\{\psi \in V_\mathbb{H}^L~|~T\phi =\psi \mbox{~for}~~\phi \in\mathcal{D}(T)\}\\
	\ker(T)&=&\{\phi \in\mathcal{D}(T)~|~T\phi =0\}.
\end{eqnarray*}
We call an operator $T\in \mathcal{L}(V_\mathbb{H}^L)$ is bounded if
\begin{equation*}
\|T\|=\sup_{\|\phi \|=1}\|T\phi \|<\infty.
\end{equation*}
or equivalently, there exists $M\geq 0$ such that $\|T\phi \|\leq M\|\phi \|$ for $\phi \in\mathcal{D}(T)$. The set of all bounded left linear operators in $V_\mathbb{H}^L$ will be denoted by $\mathcal{B}(V_\mathbb{H}^L)$.\\
\begin{proposition}\cite{Fa}
	Let $T\in\B(V_\mathbb{H}^L)$. Then $T$ is self-adjoint if and only if for each $\phi\in V_{\mathbb{H}}^{L}$, $\langle T\phi\mid\phi\rangle\in\mathbb{R}$. 
\end{proposition}
\begin{definition}\label{Def1}
	Let $T_{1}$ and $T_{2}$ be self-adjoint operators on $V_\mathbb{H}^L$. Then $T_{1}\leq T_{2}$ ($T_{1}$ less or equal to $T_{2}$) or equivalently $T_{2}\geq T_{1}$ if $\langle T_{1}\phi|\phi\rangle\leq\langle T_{2}\phi|\phi\rangle$ for all $\phi\in V_\mathbb{H}^L$. In particular $T_{1}$ is called positive if $T_{1}\geq 0$ or $\left\langle T_{1}f|f\right\rangle \geq 0,$  for all $f\in V_\mathbb{H}^{L}. $
\end{definition}
\begin{theorem}\cite{Ric}\label{IT1}
	Let $A\in\B(V_\mathbb{H}^L)$. If $A\geq 0$ then there exists a unique operator in $\B(V_\mathbb{H}^L)$, indicated by $\sqrt{A}=A^{1/2}$ such that $\sqrt{A}\geq 0$ and $\sqrt{A}\sqrt{A}=A$.
\end{theorem}
\begin{proposition}\cite{Ric} Let $A\in\B(V_\mathbb{H}^L)$. If $A\geq 0$, then $A$ is self-adjoint.
	\end{proposition}
\begin{lemma}\label{L3}
	Let $U_{\mathbb{H}}^{L}$ and $V_{\mathbb{H}}^{L}$ be left quaternion Hilbert spaces. Let $T:\D(A)\longrightarrow V_{\mathbb{H}}^{L}$ be a linear operator with domain $\D(T)\subseteq U_{\mathbb{H}}^{L} $ and $\mbox{ran}(T)\subseteq V_{\mathbb{H}}^{L},$ then the inverse $T^{-1}:\mbox{ran}(T)\longrightarrow \D(T) $ exists if and only if $T\phi=0\Rightarrow \phi=0.$
\end{lemma}
\begin{proof}Since the non-commutativity of quaternions does not play a role in the proof, it follows from its complex counterpart.
\end{proof}\noindent
\begin{lemma}\label{L2}
	Let $T\in\B(V_\mathbb{H}^L)$ be a self-adjoint operator, then
	\begin{equation}\label{E13} 
	\left\|T\right\|=\sup_{\left\|\phi\right\|=1}\left|\left\langle \phi|T\phi \right\rangle\right|
	\end{equation} 		 
\end{lemma}
\begin{proof}
	In fact, the non-commutativity of quaternions does not play a role in the proof, the proof follows from its complex counterpart. One may also see \cite{Fa} where a proof is given.
\end{proof}\noindent
We define $\mathcal{GL}(V_{\mathbb{H}}^{L}),$ the set of all bounded linear operators in $V_{\mathbb{H}}^{L}$ with bounded inverse.  $$\mathcal{GL}(V_{\mathbb{H}}^{L})=\left\{T:V_{\mathbb{H}}^{L}\longrightarrow V_{\mathbb{H}}^{L}: T\mbox{~bounded and~} T^{-1}\mbox{~bounded}\right\}.$$
Also   $\mathcal{GL}^{+}(V_{\mathbb{H}}^{L}) $ is the set of positive operators in $\mathcal{GL}(V_{\mathbb{H}}^{L}).$
\begin{proposition}\label{pro1}
	Let $T:V_{\mathbb{H}}^{L}\longrightarrow V_{\mathbb{H}}^{L}$ be a left linear operator. Then the following are equivalent statements:
	\begin{itemize}
		\item [i.]There exist $m>0$ and $M<\infty$ such that $mI_{V_{\mathbb{H}}^{L}}\leq T\leq MI_{V_{\mathbb{H}}^{L}};$
		\item[ii.] $T$ is positive and there exist $m>0$ and $M<\infty$ such that $m\left\| f\right\|^{2}\leq \left\| T^{\frac{1}{2}}f\right\|^{2} \leq M\left\| f\right\|^{2}; $
		\item [iii.] $T$ is positive and $T^{\frac{1}{2}}\in \mathcal{GL}(V_{\mathbb{H}}^{L}) ;$
		\item[iv.]There exists a self-adjoint operator $A\in \mathcal{GL}(V_{\mathbb{H}}^{L})$ such that $A^{2}=T;$
		\item[v.] $T\in \mathcal{GL}^{+}(V_{\mathbb{H}}^{L}).$
	\end{itemize}
\end{proposition}
\begin{proof}
For (i.)$\Rightarrow$(ii.), suppose that  $m>0$ and $M<\infty$ such that $mI_{V_{\mathbb{H}}^{L}}\leq T\leq MV_{\mathbb{H}}^{L}.$
Let $f\in V_{\mathbb{H}}^{L}$ then $\left\langle mI_{V_{\mathbb{H}}^{L}}f|f \right\rangle \leq\left\langle Tf|f \right\rangle  \leq \left\langle MI_{V_{\mathbb{H}}^{L}}f|f \right\rangle  .$ It follows that 
\begin{equation}\label{s21}
m\left\| f\right\|^{2}\leq \left\langle Tf|f\right\rangle\leq M\left\|f\right\|^{2}  
\end{equation}	
Hence $\left\langle Tf|f\right\rangle \geq m\left\| f\right\|^{2}\geq 0,$ as $m>0.$ Thereby $\left\langle Tf|f\right\rangle \geq 0$ and $T$ is positive.\\
Since $T$ is positive, from Theorem \ref{IT1}, there exists a unique operator in $\B(V_\mathbb{H}^L)$, indicated by $\sqrt{T}=T^{1/2}$ such that $\sqrt{T}\geq 0$ and $\sqrt{T}\sqrt{T}=T.$\\
Now equation \ref{s21} becomes
\begin{equation*}
m\left\| f\right\|^{2}\leq \left\langle T^{1/2}T^{1/2}f|f\right\rangle\leq M\left\|f\right\|^{2}.  
\end{equation*}
It follows that
\begin{equation*}
m\left\| f\right\|^{2}\leq \left\langle T^{1/2}f|(T^{1/2})^{\dagger}f\right\rangle\leq M\left\|f\right\|^{2}  
\end{equation*}
and
\begin{equation*}
m\left\| f\right\|^{2}\leq \left\langle T^{1/2}f|T^{1/2}f\right\rangle\leq M\left\|f\right\|^{2},\mbox{~as~}  T^{1/2} \mbox{~is positive.~}
\end{equation*}
Thereby $m\left\| f\right\|^{2}\leq \left\| T^{\frac{1}{2}}f\right\|^{2} \leq M\left\| f\right\|^{2}.$\\
For (ii)$\Rightarrow$(i), suppose that  $T$ is positive and there exists $m>0$ and $M<\infty$ such that $m\left\| f\right\|^{2}\leq \left\| T^{\frac{1}{2}}f\right\|^{2} \leq M\left\| f\right\|^{2}.$ Then
\begin{equation*}
m\left\| f\right\|^{2}\leq \left\langle T^{1/2}f|T^{1/2}f\right\rangle\leq M\left\|f\right\|^{2}.
\end{equation*}
It follows that
\begin{equation*}
m\left\| f\right\|^{2}\leq \left\langle (T^{1/2})^{\dagger}T^{1/2}f|f\right\rangle\leq M\left\|f\right\|^{2}  
\end{equation*}
and
\begin{equation*}
m\left\| f\right\|^{2}\leq \left\langle T^{1/2}T^{1/2}f|f\right\rangle\leq M\left\|f\right\|^{2},\mbox{~as~}  T^{1/2} \mbox{~is positive.~}
\end{equation*}
Hence 
\begin{equation*}
m\left\| f\right\|^{2}\leq \left\langle Tf|f\right\rangle\leq M\left\|f\right\|^{2}.  
\end{equation*}	
Therefore $mI_{V_{\mathbb{H}}^{L}}\leq T\leq M I_{V_{\mathbb{H}}^{L}}.$\\
For (ii.)$\Rightarrow$ (iii.), suppose that $T$ is positive and there exists $m>0$ and $M<\infty$ such that $m\left\| f\right\|^{2}\leq \left\| T^{\frac{1}{2}}f\right\|^{2} \leq M\left\| f\right\|^{2}.$ Since $T$ is positive, from Theorem \ref{IT1}, $T^{\frac{1}{2}}$ is bounded. Let $f\in V_\mathbb{H}^L,$ assume that $ T^{\frac{1 }{2}}f=0$ then
$m\left\| f\right\|^{2}\leq \left\| 0\right\|^{2} \leq M\left\| f\right\|^{2}.$ It follows that $f=0,$ from Lemma \ref{L3},$(T^{\frac{1}{2}})^{-1}:V_\mathbb{H}^L\longrightarrow V_\mathbb{H}^L$ exists. For $f\in V_\mathbb{H}^L,$ there exists $g\in V_\mathbb{H}^L$ such that $T^{\frac{1}{2}}f=g.$ That is $f=(T^{\frac{1}{2}})^{-1}g.$ Now $m\left\| f\right\|^{2}\leq \left\| T^{\frac{1}{2}}f\right\|^{2}$ implies
\begin{eqnarray*}
	m\left\| (T^{\frac{1}{2}})^{-1}g\right\|^{2}&\leq &\left\| g\right\|^{2}
\end{eqnarray*} 
and
\begin{eqnarray*}
	\left\| (T^{\frac{1}{2}})^{-1}g\right\|&\leq &\frac{1}{\sqrt{m}}\left\| g\right\|.
\end{eqnarray*} 
It follows that $(T^{\frac{1}{2}})^{-1}$ is bounded and hence $T^{\frac{1}{2}}\in \mathcal{GL}(V_{\mathbb{H}}^{L}). $\\
For (iii)$\Rightarrow$(ii), suppose that $T$ is positive and $T^{\frac{1}{2}}\in \mathcal{GL}(V_{\mathbb{H}}^{L}). $ Then $T^{\frac{1}{2}}$ is bounded and  $(T^{\frac{1}{2}})^{-1}$ is also bounded. Therefore, one may conclude that there exists $m>0$ and $M<\infty$ such that $m\left\| f\right\|^{2}\leq \left\| T^{\frac{1}{2}}f\right\|^{2} \leq M\left\| f\right\|^{2}.$\\
For (iii.)$\Rightarrow$(iv.), assume that $T$ is positive and $T^{\frac{1}{2}}\in \mathcal{GL}(V_{\mathbb{H}}^{L}).$ 
 Since $T$ is positive, from Theorem \ref{IT1}, there exists a unique operator in $\B(V_\mathbb{H}^L)$, indicated by $\sqrt{T}=T^{1/2}$ such that $\sqrt{T}\geq 0$ and $(\sqrt{T})^{2}=T.$ If we take $A=T^{1/2}$ then $A$ is self adjoint as $\sqrt{T}\geq 0$  and $A\in \mathcal{GL}(V_{\mathbb{H}}^{L}). $ Hence there exists a self-adjoint operator $A\in \mathcal{GL}(V_{\mathbb{H}}^{L})$ such that $A^{2}=T.$\\
 For (iv.)$\Rightarrow$(iii.), suppose that there exists a self-adjoint operator $A\in \mathcal{GL}(V_{\mathbb{H}}^{L})$ such that $A^{2}=T.$ If we take $A=T^{1/2}$ then $T^{\frac{1}{2}}\in \mathcal{GL}(V_{\mathbb{H}}^{L})$ and $T$ is positive.\\  
 For (iv.)$\Rightarrow$ (v.), assume that there exists a self-adjoint operator $A\in \mathcal{GL}(V_{\mathbb{H}}^{L})$ such that $A^{2}=T.$ Then clearly $T\in \mathcal{GL}^{+}(V_{\mathbb{H}}^{L}).$ \\
 For (v.)$\Rightarrow$ (iv.), assume that $T\in \mathcal{GL}^{+}(V_{\mathbb{H}}^{L}).$ Then $T$ is positive and bounded. From Theorem \ref{IT1}, there exists a unique operator $A(=T^{1/2})$ in $\mathcal{GL}(V_\mathbb{H}^L)$ such that $A\geq 0$ and $A.A=T$. It follows that there exists a self-adjoint operator $A\in \mathcal{GL}(V_{\mathbb{H}}^{L})$ such that $A^{2}=T.$
\end{proof}
\subsection{Frames and Frame operators}
Let $V_\mathbb{H}^L$ be a finite dimensional left quaternion Hilbert space. A countable family of elements $\left\{\phi_{k}\right\}_{k\in I}$ in $V_\mathbb{H}^L$  is a frame for $V_\mathbb{H}^L$  if there exist constants $A,B>0$ such that
\begin{equation}\label{eq5}
A\left\|\psi\right\|^{2}\leq\displaystyle\sum_{k\in I}\left|\left\langle \psi|\phi_{k}\right\rangle\right|^{2}\leq B\left\|\psi\right\|^{2},\quad~\text{for all}~ \psi\in V^{R}_{H}.
\end{equation}
The numbers $A$ and $B$ are called frame bounds. They are not unique. The \textit{optimal lower frame bound }is the supremum over all lower frame bounds, and the \textit{optimal upper frame bound} is the infimum over all upper frame bounds.  The frame $\left\{\phi_{k}\right\}_{k\in I}$ is said to be normalized if $\left\|\phi_{k}\right\|=1,~\mbox{\text{for all}~} k\in I.$ \\
Let $\left\{\phi_{k}\right\}_{k\in I}$ be  a frame  on a left quaternionic Hilbert space $V_\HQ^L$ and define a linear mapping 
\begin{equation}\label{eq15}
T:\HQ^{|I|}\longrightarrow V^{L} _{\mathbb{H}},~T\left\{q_{k}\right\}_{k\in I}=\displaystyle\sum_{k\in I} q_{k}\phi_{k},~q_{k}\in \HQ,
\end{equation} 
where $|I|$ is the cardinality of $I$. $T$ is usually called the \textit{pre-frame operator}, or the \textit{synthesis operator.}\\
The adjoint operator
\begin{equation}\label{eq16}
T^{\dagger}:V^{L}_{\mathbb{H}}\longrightarrow \HQ^{|I|}, \mbox{~~by~~}T^{\dagger}\psi=\left\{\left\langle \psi|\phi_{k}\right\rangle\right\}_{k\in I}	
\end{equation}
is called the \textit{analysis operator.}\\
By composing $T$ with its adjoint we obtain the \textit{frame operator}
\begin{equation}\label{eq17}
S:V^{L}_{\mathbb{H}}\longrightarrow V^{L} _{\mathbb{H}},~S\psi=TT^{\dagger}\psi=\displaystyle\sum_{k\in I}\left\langle \psi|\phi_{k}\right\rangle\phi_{k}	.
\end{equation}
Note that in terms of the frame operator, for $\psi\in V^{L} _{\mathbb{H}}$
\begin{eqnarray*}
	\left\langle S\psi|\psi\right\rangle =\left\langle \sum_{k=1}^{m}\left\langle \psi|\phi_{k}\right\rangle \phi_{k}\middle|\psi\right\rangle
	=\sum_{k=1}^{m}\left\langle \psi|\phi_{k}\right\rangle\left\langle \phi_{k}|\psi\right\rangle
	=\sum_{k=1}^{m}\left|\left\langle \psi|\phi_{k}\right\rangle\right|^{2}.
\end{eqnarray*}
\begin{proposition}\cite{KH}\label{KK}
Let $\left\{\phi_{k}\right\}_{k\in I}$ be a frame for $V^{L}_{\HQ}$ with frame operator $S.$ Then
\begin{itemize}
\item[i.] $S$ is invertible and self-adjoint.
\item[ii.] Every $\psi\in V_{\HQ}^{L},$ can be represented as
	$$\psi=\sum_{k\in I}\left\langle \psi|S^{-1}\phi_{k}\right\rangle \phi_{k}=\sum_{k\in I}\left\langle \psi|\phi_{k}\right\rangle S^{-1} \phi_{k}.$$
\end{itemize}
\end{proposition}
\section{Controlled Frames in Left Quaternion Hilbert Spaces}
\noindent
Controlled frames were first introduced for spherical wavelets  in \cite{Bog} to get a numerically more efficient approximation algorithm. For general frames, it was developed in \cite{Bal4}.   In this section the concept of controlled frames in quaternionic Hilbert spaces, along the lines of the arguments of \cite{Bal4}, is presented. 
\begin{definition}\label{DEF1}
Let $\mathfrak{C}\in \mathcal{GL}(V_{\mathbb{H}}^{L}).$ A countable family of vectors $\Phi=\{\phi_{k}\in V^{L}_{\HQ} : k\in I \}$	is said to be a frame controlled by the operator $\mathfrak{C}$ or the $\mathfrak{C}-$ controlled frame if there exist  constants $0<A _{\mathfrak{C} S}\leq B_{\mathfrak{C} S}<\infty$  such that 

\begin{equation}\label{CFDE}
A _{\mathfrak{C} S}\left\|\psi \right\|^{2} \leq\displaystyle\sum_{k\in I}\langle\psi|\phi_{k}\rangle\langle \mathfrak{C} \phi_{k}|\psi\rangle\leq B_{\mathfrak{C} S}\left\|\psi \right\|^{2} , 
\end{equation}
for all $\psi\in V_{\mathbb{H}}^{L}. $
\end{definition}\noindent
Controlled frame operator can be defined as 
\begin{equation}
S_{\mathfrak{C}}\psi=\displaystyle\sum_{k\in I}\langle \psi|\phi_{k}\rangle\mathfrak{C}\phi_{k}.
\end{equation}
Now we have the frame operator
\begin{equation}
S:V^{L}_{\mathbb{H}}\longrightarrow V^{L} _{\mathbb{H}},~S\psi=TT^{\dagger}\psi=\displaystyle\sum_{k\in I}\left\langle \psi|\phi_{k}\right\rangle\phi_{k}	.
\end{equation}
For $\mathfrak{C}:V_{\mathbb{H}}^{L}\longrightarrow V_{\mathbb{H}}^{L},$ consider
\begin{eqnarray*}
\left\langle\mathfrak{C}S\psi|\psi \right\rangle &=&\left\langle \mathfrak{C}\left(\sum_{k\in I}\left\langle \psi|\phi_{k}\right\rangle\phi_{k}	 \right)|\psi \right\rangle\\
&= & \left\langle \sum_{k\in I}\left\langle \psi|\phi_{k}\right\rangle\mathfrak{C}\phi_{k}|\psi\right\rangle\\
&=&\sum_{k\in I}\left\langle \psi|\phi_{k}\right\rangle\left\langle \mathfrak{C}\phi_{k}|\psi\right\rangle.
\end{eqnarray*}
Now equation \ref{CFDE} becomes
\begin{equation}\label{CFDE1}
A _{\mathfrak{C} S}\left\|\psi \right\|^{2} \leq\left\langle\mathfrak{C}S\psi|\psi \right\rangle\leq B_{\mathfrak{C} S}\left\|\psi \right\|^{2} , 
\end{equation}
for all $\psi\in V_{\mathbb{H}}^{L}. $
That is, there exist  constants $0<A _{\mathfrak{C} S}\leq B_{\mathfrak{C} S}<\infty$  such that 
\begin{equation*}
A _{\mathfrak{C} S}I_{V^{L}_{\mathbb{H}}}\leq\mathfrak{C} S\leq B _{\mathfrak{C} S}I_{V^{L}_{\mathbb{H}}}.
\end{equation*}
From Proposition \ref{pro1}, $\mathfrak{C}S\in \mathcal{GL}^{+}(V_{\mathbb{H}}^{L}), $
and the definition \ref{DEF1} is clearly equivalent to $\mathfrak{C}S\in \mathcal{GL}^{+}(V_{\mathbb{H}}^{L}). $
\begin{proposition}\label{NC}
Let $\Phi=\{\phi_{k}\in V^{L}_{\HQ} : k\in I \}$ be a $\mathfrak{C}-$ controlled frame in $V^{L}_{\mathbb{H}}$ for $\mathfrak{C}\in \mathcal{GL}(V_{\mathbb{H}}^{L}).$ Then $\Phi$ is a frame in $V^{L}_{\mathbb{H}}.$ Moreover $\mathfrak{C} S=S\mathfrak{C}^{\dagger}$  and $$\displaystyle\sum_{k\in I}\left\langle \psi|\phi_{k}\right\rangle\mathfrak{C} \phi_{k}=\displaystyle\sum_{k\in I}\left\langle \psi|\mathfrak{C}\phi_{k}\right\rangle\phi_{k},$$
for all $\psi\in V_{\mathbb{H}}^{L}. $	
\end{proposition}
\begin{proof}
Let $\left\{\phi_{k}\right\}_{k\in I}$ be a $\mathfrak{C}-$ controlled frame in $V^{L}_{\mathbb{H}}$ for $\mathfrak{C}\in \mathcal{GL}(V_{\mathbb{H}}^{L}).$ Then there exist  constants $0<A _{\mathfrak{C} S}\leq B_{\mathfrak{C} S}<\infty$  such that 
\begin{equation}
A _{\mathfrak{C} S}\left\|\psi \right\|^{2} \leq\displaystyle\sum_{k\in I}\langle\psi|\phi_{k}|\rangle\langle \mathfrak{C} \phi_{k}|\psi\rangle\leq B_{\mathfrak{C} S}\left\|\psi \right\|^{2} , 
\end{equation}
for all $\psi\in V_{\mathbb{H}}^{L}. $ It follows that

\begin{equation}
A _{\mathfrak{C} S}\left\langle I_{V^{L}_{\mathbb{H}}}\psi|\psi \right\rangle  \leq\left\langle S_{\mathfrak{C}}\psi|\psi\right\rangle\leq B_{\mathfrak{C} S}\left\langle I_{V^{L}_{\mathbb{H}}}\psi|\psi \right\rangle  , 
\end{equation}
for all $\psi\in V_{\mathbb{H}}^{L}. $\\
 Hence $A_{\mathfrak{C} S}I_{V^{L}_{\mathbb{H}}} \leq  S_{\mathfrak{C}} \leq B_{\mathfrak{C} S}I_{V^{L}_{\mathbb{H}}}.$ From Proposition \ref{pro1}, $S_{\mathfrak{C}}\in \mathcal{GL}^{+}(V_{\mathbb{H}}^{L}). $\\
 Define $\widehat{S}=\mathfrak{C}^{-1}S_{\mathfrak{C}}.$  Then $\widehat{S}\in \mathcal{GL}(V_{\mathbb{H}}^{L})$ as  $\mathfrak{C}^{-1},S_{\mathfrak{C}}\in \mathcal{GL}(V_{\mathbb{H}}^{L}). $\\
Let $\psi\in V_{\mathbb{H}}^{L} $ then
\begin{eqnarray*}
\widehat{S}\psi &=&\mathfrak{C}^{-1}S_{\mathfrak{C}}\psi\\
&=&\mathfrak{C}^{-1}\left(\sum_{k\in I}\langle \psi|\phi_{k}\rangle\mathfrak{C}\phi_{k} \right) \\
&=&\sum_{k\in I}\langle \psi|\phi_{k}\rangle\mathfrak{C}^{-1}\mathfrak{C}\phi_{k}\\
&=&\sum_{k\in I}\langle \psi|\phi_{k}\rangle I_{V^{L}_{\mathbb{H}}}\phi_{k}\\
&=&\sum_{k\in I}\langle \psi|\phi_{k}\rangle \phi_{k}\\
&=&S\psi
\end{eqnarray*}
Hence $S$ is everywhere defined and $S\in \mathcal{GL}(V_{\mathbb{H}}^{L}).$ Thereby $\Phi$ is a frame in $V^{L}_{\mathbb{H}}.$\\
Since $S_{\mathfrak{C}}\in \mathcal{GL}^{+}(V_{\mathbb{H}}^{L}),~S_{\mathfrak{C}}$ is positive, therefore $S_{\mathfrak{C}}$
is self-adjoint.\\
For $\psi\in V^{L}_{\mathbb{H}}, $
\begin{eqnarray*}
\mathfrak{C} S\psi&=&\mathfrak{C}\left( \sum_{k\in I}\left\langle \psi|\phi_{k}\right\rangle\phi_{k}	\right)\\
&=& \sum_{k\in I}\left\langle \psi|\phi_{k}\right\rangle \mathfrak{C}\phi_{k}\\
&=&S_{\mathfrak{C}}\psi.
\end{eqnarray*}
Hence $\mathfrak{C} S=S_{\mathfrak{C}}.$\\
Now 
$S_{\mathfrak{C}}^{\dagger}=(\mathfrak{C} S)^{\dagger}=S^{\dagger}\mathfrak{C}^{\dagger}=S\mathfrak{C}^{\dagger}.$ But $S_{\mathfrak{C}}^{\dagger}=S_{\mathfrak{C}},$ as $S_{\mathfrak{C}}$ is self-adjoint. Therefore $S_{\mathfrak{C}}=S\mathfrak{C}^{\dagger}.$ 
It follows that $ \mathfrak{C} S=S\mathfrak{C}^{\dagger}.$\\
Also for $\psi\in V^{L}_{\mathbb{H}},  ~\mathfrak{C} S\psi=S\mathfrak{C}^{\dagger}\psi$ and $S_{\mathfrak{C}}\psi=S\mathfrak{C}^{\dagger}\psi.$ 
Hence
\begin{eqnarray*}
\sum_{k\in I}\left\langle \psi|\phi_{k}\right\rangle\mathfrak{C} \phi_{k}&=&S_{\mathfrak{C}}\psi\\
&=&S\mathfrak{C}^{\dagger}\psi\\
&=&\sum_{k\in I}\left\langle \mathfrak{C}^{\dagger}\psi|\phi_{k}\right\rangle\phi_{k}\\
&=&\sum_{k\in I}\left\langle \psi|\mathfrak{C}\phi_{k}\right\rangle\phi_{k}		
\end{eqnarray*}
Thereby $\displaystyle\sum_{k\in I}\left\langle \psi|\phi_{k}\right\rangle\mathfrak{C} \phi_{k}=\displaystyle\sum_{k\in I}\left\langle \psi|\mathfrak{C}\phi_{k}\right\rangle\phi_{k}.$
\end{proof}\noindent
The above Proposition shows that every controlled frame is a usual frame. But if $\mathfrak{C}\in \mathcal{GL}(V_{\mathbb{H}}^{L})$ is self-adjoint, we can give a necessary and sufficient condition for a frame to be a controlled frame and vice-versa.

\begin{proposition}\label{CFPro}
Let $\mathfrak{C}\in \mathcal{GL}(V_{\mathbb{H}}^{L})$ be self-adjoint. The family $\left\{\phi_{k}\right\}_{k\in I}$  is a frame controlled by $\mathfrak{C}$ if and only if it is a frame in $V_{\mathbb{H}}^{L}$ and $\mathfrak{C}$ is positive and commutes with the frame operator $S.$
\end{proposition}
\begin{proof}
Let $\mathfrak{C}\in \mathcal{GL}(V_{\mathbb{H}}^{L})$ be self-adjoint.  Suppose that $\left\{\phi_{k}\right\}_{k\in I}$  is a frame controlled by $\mathfrak{C}.$  Then from Proposition \ref{NC}, $\left\{\phi_{k}\right\}_{k\in I}$ is a frame in $V^{L}_{\mathbb{H}}$  and  $\mathfrak{C}S=S\mathfrak{C}^{\dagger}.$ Since $\mathfrak{C}$ is self-adjoint, $\mathfrak{C}=\mathfrak{C}^{\dagger}.$	Hence $\mathfrak{C}S=S\mathfrak{C}.$ Thereby $\mathfrak{C}$ commutes with the  frame operator $S.$  It follows that $\mathfrak{C}=S\mathfrak{C} S^{-1}=S_{\mathfrak{C}}S^{-1}$ and  $\mathfrak{C}$ is positive.\\
On the other hand suppose that $\left\{\phi_{k}\right\}_{k\in I}$  is a frame  in $V_{\mathbb{H}}^{L}$ with frame operator $S$ and $\mathfrak{C}$ is positive and commutes with $S.$  Then $S\in \mathcal{GL}^{+}(V_{\mathbb{H}}^{L}).$ Therefore $\mathfrak{C} S=S_{\mathfrak{C}}\in \mathcal{GL}^{+}(V_{\mathbb{H}}^{L})$ and so $S_{\mathfrak{C}}$ is positive.\\
From Proposition \ref{pro1}, there exist $A>0$ and $B<\infty$ such that 
\begin{equation}\label{con}
AI_{V_{\mathbb{H}}^{L}}\leq S_{\mathfrak{C}}\leq BI_{V_{\mathbb{H}}^{L}}.
\end{equation} 
For $\psi\in V_{\mathbb{H}}^{L}, $ \ref{con} becomes
\begin{equation}
\langle AI_{V_{\mathbb{H}}^{L}}\psi|\psi\rangle\leq \langle S_{\mathfrak{C}}\psi|\psi\rangle\leq \langle BI_{V_{\mathbb{H}}^{L}}\psi|\psi\rangle.
\end{equation}
It follows that
\begin{equation}
A\left\| \psi\right\| ^{2}\leq \sum_{k\in I}\langle \psi|\phi_{k}\rangle\left\langle \mathfrak{C}\phi_{k}|\psi\right\rangle \leq B\left\| \psi\right\| ^{2}
\end{equation}
Hence $\left\{\phi_{k}\right\}_{k\in I}$  is a frame controlled by $\mathfrak{C}.$ 
\end{proof}
\section{Weighted frames and frame multipliers in $V_{\mathbb{H}}^L$}\noindent
In this section we present  connection between frame multipliers and weighted frames in a left quaternionic Hilbert space.
\begin{definition}
Let $\{\phi_{k}\}_{k\in I}$ be a sequence of elements in $V_{\mathbb{H}}^{L}$ and $\{\omega_{k}\}_{k\in I}\subseteq \mathbb{R}^{+}$ a sequence of positive weights.This pair is called a $\omega-$ frame for $V_{\mathbb{H}}^{L}$  if there exist constants $A>0$ and $B<\infty$ such that
	\begin{equation}
	A\left\|\psi \right\|^{2}\leq \sum \omega_{k}\left|\left\langle \psi|\phi_{k} \right\rangle \right|^{2}\leq   B\left\| \psi\right\|^{2}, 
	\end{equation}
	for all $\psi\in V_{\mathbb{H}}^{L}.$\\

\end{definition}
\begin{definition}
	A sequence $\{\zeta_{n}\}$ is called semi-normalized if there are bounds $b\geq a>0$ such that $a\leq\left| \zeta_{n}\right| \leq b.$
\end{definition}
\begin{definition}
Let $ U_{\mathbb{H}}^{L},V_{\mathbb{H}}^{L}$ be left quaternionic Hilbert spaces, let $\{\psi_{k}\}_{k\in I}\subseteq U_{\mathbb{H}}^{L}$ and $\{\phi_{k}\}_{k\in I}\subseteq V_{\mathbb{H}}^{L}$ be frames. Fix  $\mathbf{m}=\{m_{k}\}\in \ell^{\infty}(I).$ 	 Define the operator $\mathbf{M}_{\mathbf{m},\{\phi_{k}\},\{\psi_{k}\}}:U_{\mathbb{H}}^{L}\longrightarrow V_{\mathbb{H}}^{L}$ the \textit{frame multiplier} for the frames $\{\psi_{k}\}$ and $\{\phi_{k}\}$, as the operator
\begin{equation}
\mathbf{M}_{\mathbf{m},\{\phi_{k}\},\{\psi_{k}\}}(h)=\sum_{k\in I}m_{k}\left\langle h|\psi_{k} \right\rangle \phi_{k};\quad h\in U_{\mathbb{H}}^{L}.
\end{equation}
The sequence $\mathbf{m}$ is called the {\em symbol} of $\mathbf{M}.$ We will denote 
$\mathbf{M}_{\mathbf{m},\{\phi_{k}\}}=\mathbf{M}_{\mathbf{m},\{\phi_{k}\},\{\phi_{k}\}}$
\end{definition}
\begin{proposition}
Let $\mathfrak{C}\in  \mathcal{GL}(V_{\mathbb{H}}^{L}) $ be self-adjoint and diagonal on $\Phi=\{\phi_{k}\}_{k\in I}$ and assume it generates a controlled frame. Then the sequence $\{\omega_{k}\}$ which verifies the relation $\mathfrak{C}\phi_{k}=\omega_{k}\phi_{k}$ is semi normalized and positive. Furthermore $\mathfrak{C}=M_{\omega,\widetilde{\Phi},\Phi},$ where $\widetilde{\Phi}=\{L^{-1}\phi_k\}_{k\in I}$ and $L$ is the frame operator for $\Phi$.
\end{proposition}
\begin{proof}
Since $\mathfrak{C}\in  \mathcal{GL}(V_{\mathbb{H}}^{L}) $ is self-adjoint and $\Phi=\{\phi_{k}\}$  is a frame controlled by the operator $\mathfrak{C}$, from Proposition \ref{CFPro}, $\mathfrak{C}$ is positive.  Thereby $\mathfrak{C}\in \mathcal{GL}^{+}(V_{\mathbb{H}}^{L}).$ Since  $\mathfrak{C}\in \mathcal{GL}^{+}(V_{\mathbb{H}}^{L}),$ from Proposition \ref{pro1}, there exists $A>0$ and $B<\infty$ such that
\begin{equation}\label{CFM}
A\left\|\psi\right\|^{2}\leq\left\|\mathfrak{C}^{\frac{1}{2}} \psi\right\|^{2} \leq
B\left\|\psi\right\|^{2},
\end{equation}
for all $ \psi\in V_{\mathbb{H}}^{L}.$ Since $\mathfrak{C}\phi_{k}=\omega_{k}\phi_{k},~ \mathfrak{C}^{\frac{1}{2}}\phi_{k}=\sqrt{\omega_{k}}\phi_{k}.$ Now equation \ref{CFM} gives
\begin{equation}
0<A\leq \omega_{k}\leq B.
\end{equation}
Hence the sequence $\{\omega_{k}\}$ is positive and semi-normalized.\\
Since $\mathfrak{C}\in  \mathcal{GL}(V_{\mathbb{H}}^{L}) $ is self-adjoint and $\Phi=\{\phi_{k}\}$  is a frame controlled by the operator $\mathfrak{C}$, from Proposition \ref{CFPro}, $\Phi=\{\phi_{k}\}$  is a frame in $V_{\mathbb{H}}^{L}.	$\\
Let $\psi\in V_{\mathbb{H}}^{L}$ then, by Proposition \ref{KK}, $\psi=\displaystyle\sum_{k\in I}\left\langle\psi|\phi_{k} \right\rangle L^{-1}\phi_{k}, $  where $L$ is the frame operator for $\Phi=\{\phi_{k}\}.$\\
Now
\begin{eqnarray*}
M_{\omega,\widetilde{\Phi},\Phi}\psi&=&\sum_{k\in I}\omega_{k}\left\langle \psi|\phi_{k}\right\rangle \widetilde{\phi}_{k}\\
&=&\sum_{k\in I}\left\langle \psi|\phi_{k}\right\rangle \omega_{k}\widetilde{\phi}_{k}\mbox{~as~}\omega_{k}\mbox{~is real}\\
&=&\sum_{k\in I}\left\langle \psi|\phi_{k}\right\rangle \mathfrak{C}\widetilde{\phi}_{k}\\
&=&\mathfrak{C}\left(\sum_{k\in I}\left\langle \psi|\phi_{k}\right\rangle\widetilde{\phi}_{k} \right) \\
&=&\mathfrak{C}\left(\sum_{k\in I}\left\langle \psi|\phi_{k}\right\rangle L^{-1}\phi_{k} \right) \\
&=&\mathfrak{C}\psi.
\end{eqnarray*}
Hence $\mathfrak{C}=M_{\omega,\widetilde{\Phi},\Phi}.$
\end{proof}
\begin{lemma}\label{WL1}
Let $\{\omega_{k}\}$ be a semi normalized sequence with bounds $a\mbox{~and~} b.$ If $\{\phi_{k}\}$ is a frame with bounds $A\mbox{~and~}B$ then $\{\omega_{k}\phi_{k}\}$ is also a frame with bounds $a^{2}A\mbox{~and~}b^{2}B.$
\end{lemma}
\begin{proof}
Since $\{\omega_{k}\}$ is semi normalized, there exists $b\geq a>0$ such that
\begin{equation}
a\leq\left| \omega_{k}\right| \leq b.
\end{equation} 
Since $\{\phi_{k}\}$ is a frame with bounds $A\mbox{~and~}B,$
\begin{equation}
A\left\| \psi\right\|^{2}\leq 	\sum_{k\in I} \left|\left\langle \psi|\phi_{k} \right\rangle  \right|^{2}\leq B\left\| \psi\right\|^{2},
\end{equation}
for all $\psi\in V_{\mathbb{H}}^{L}.$\\ 
Let $\psi\in V_{\mathbb{H}}^{L}$ then $\left|\left\langle \psi|\omega_{k}\phi_{k} \right\rangle  \right|^{2}=\left|\omega_{k}\right|^{2}\left|\left\langle \psi|\phi_{k}\right\rangle \right|^{2}.$ Thereby
\begin{eqnarray*}
\sum_{k\in I} \left|\left\langle \psi|\omega_{k}\phi_{k} \right\rangle  \right|^{2}&=&	\sum_{k\in I} \left|\omega_{k}\right|^{2}\left|\left\langle \psi|\phi_{k}\right\rangle \right|^{2}\\
	&\leq & b^{2}	\sum_{k\in I} \left|\left\langle \psi|\phi_{k}\right\rangle \right|^{2}\\
	&\leq&b^{2}B\left\| \psi\right\| ^{2}
\end{eqnarray*}	
Similarly one can prove that $\displaystyle	\sum_{k\in I} \left|\left\langle \psi|\omega_{k}\phi_{k} \right\rangle  \right|^{2}\geq a^{2}A\left\| \psi\right\| ^{2}.$\\
Hence $a^{2}A\left\| \psi\right\| ^{2}\leq\displaystyle	\sum_{k\in I} \left|\left\langle \psi|\omega_{k}\phi_{k} \right\rangle  \right|^{2}\leq b^{2}B\left\| \psi\right\| ^{2}, \mbox{~for all~} \psi\in V_{\mathbb{H}}^{L}.$\\
It follows that $\{\omega_{k}\phi_{k}\}$ is a frame in $V_{\mathbb{H}}^{L}$ with frame bounds $a^{2}A$ and $b^{2}B.$  
\end{proof}
\begin{lemma}\label{WL2}
Let $\Phi=\{\phi_k\}$ be a frame for $V_{\mathbb{H}}^{L}.$ Let $\mathbf{m}=\{m_{k}\}$ be a positive semi-normalized sequence. Then the multiplier $\mathbf{M}_{\mathbf{m},\Phi}$ is the frame operator of the frame $\{\sqrt{m_{k}}\phi_{k}\}$ and therefore it is positive, self-adjoint and invertible. If $\{m_{k}\}$ is negative and semi-normalized then $\mathbf{M}_{\mathbf{m},\Phi}$ is negative, self-adjoint and invertible.
\end{lemma}
\begin{proof}
We have the frame multiplier for the frame $\Phi=\{\phi_{k}\}$ 
\begin{equation}
\mathbf{M}_{\mathbf{m},\Phi}\psi=\displaystyle\sum m_{k}\left\langle \psi|\phi_{k} \right\rangle \phi_{k},
\end{equation}
where $\mathbf{m}=\{m_{k}\}$ is the weight sequence.\\
Since $\{m_{k}\}$ is semi-normalized sequence and $\Phi=\{\phi_k\}$ is a frame for $V_{\mathbb{H}}^{L},$ from Lemma \ref{WL1}, $\{\sqrt{ m_{k}}\phi_{k}\}$ is a frame for $V_{\mathbb{H}}^{L}.$ \\
Let $\psi\in V_{\mathbb{H}}^{L} $ then  	
\begin{eqnarray*}
	\mathbf{M}_{\mathbf{m},\Phi}\psi&=&\sum_{k} m_{k}\left\langle \psi|\phi_{k} \right\rangle \phi_{k}\\
	&=&\sum_{k}\left\langle \psi|\sqrt{ m_{k} m_{k}}\phi_{k} \right\rangle \phi_{k}\\
	&=&\sum_{k}\left\langle \psi|\sqrt{ m_{k}}\phi_{k} \right\rangle \sqrt{ m_{k}}\phi_{k}\\
	&=&S_{\sqrt{ m_{k}}\phi_{k}}\psi,
\end{eqnarray*}
where $S_{\sqrt{ m_{k}}\phi_{k}}$ is the frame operator for the frame $\sqrt{ m_{k}}\phi_{k}.$ \\
Hence the frame multiplier $\mathbf{M}_{\mathbf{m},\Phi}$ is the frame operator of the frame $\{\sqrt{m_{k}}\phi_{k}\}.$\\
Since the frame operator is always positive, self-adjoint and invertible,  $\mathbf{M}_{\mathbf{m},\Phi}$ is positive, self-adjoint and invertible.\\
If $\{m_{k}\}$ is negative  then $m_{k}<0,~$ for all $k.$ So that $m_{k}=-\sqrt{\left|m_{k} \right|^{2 }}.$ Thereby
\begin{eqnarray*}
\mathbf{M}_{\mathbf{m},\Phi}\psi&=&-\sum_{k} \sqrt{\left|m_{k} \right|^{2 }}\left\langle \psi|\phi_{k} \right\rangle \phi_{k}\\
&=&-\sum_{k}\left\langle \psi|\sqrt{\left|m_{k} \right|^{2 }}\phi_{k} \right\rangle \phi_{k}\\
&=&-\sum_{k}\left\langle \psi|\sqrt{ \left| m_{k}\right| }\phi_{k} \right\rangle \sqrt{ \left| m_{k}\right| }\phi_{k}\\
&=&-S_{\sqrt{ \left| m_{k}\right| }\phi_{k}}\psi,
\end{eqnarray*}
Hence $\mathbf{M}_{\mathbf{m},\Phi}$ is negative, self-adjoint and invertible as the frame operator $S_{\sqrt{ \left| m_{n}\right| }\phi_{k}}$ is always positive, self-adjoint and invertible.
\end{proof}
\begin{theorem}
Let $\Phi=\{\phi_k\}$ be a sequence of elements in $V_{\mathbb{H}}^{L}.$ Let $\{\omega_{k}\}$ be a sequence of positive,semi-normalized weights. Then the following conditions are equivalent:
\begin{itemize}
\item [i.] $\{\phi_k\}$ is a frame.
\item[ii.] $\mathbf{M}_{\mathbf{m},\Phi}$ is a positive and invertible operator
\item[iii.]There are constants $A>0$ and $B<\infty$ such that $$A\left\| \psi\right\| ^{2}\leq \displaystyle\sum_{k\in I}\omega_{k}\left|\left\langle \psi|\phi_{k} \right\rangle \right|^{2}\leq B\left\| \psi\right\|^{2} .$$
\item[iv.]$\{\sqrt{\omega_{k}}\phi_{k}\}$ is a frame.
\item[v.]$\{\omega_{k}\phi_{k}\}$ is a frame. i.e., the fair $\{\omega_{k}\},\{\phi_{k}\}$ forms a weighted frame.
\end{itemize}	
\end{theorem}
\begin{proof}
For i.$\Rightarrow$ ii., suppose that $\{\phi_{k}\}$ is a frame in $V_{\mathbb{H}}^{L}.$	From Lemma \ref{WL2}, the multiplier $\mathbf{M}_{\mathbf{m},\Phi}$ is a frame operator of the frame $\{\sqrt{m_{k}}\phi_{k}\}$ and therefore it is positive and invertible.\\
 For ii.$\Leftrightarrow$ iii., suppose that	$\mathbf{M}_{\mathbf{m},\Phi}$ is a positive and invertible operator. Then  $\mathbf{M}_{\mathbf{m},\Phi}\in \mathcal{GL}^{+}(V_{\mathbb{H}}^{L}). $ From Proposition \ref{pro1}, there exists $A>0$ and $B<\infty $ such that  $AI_{V_{\mathbb{H}}^{L}}\leq \mathbf{M}_{\mathbf{m},\Phi}\leq BI_{V_{\mathbb{H}}^{L}}.$ 
 It follows that for $\psi\in V_{\mathbb{H}}^{L},$
 \begin{equation*}
 A\left\|\psi \right\| ^{2}\leq \left\langle \mathbf{M}_{\mathbf{m},\Phi}\psi|\psi \right\rangle\leq B\left\| \psi\right\| ^{2}.
 \end{equation*}
 If we take $\mathbf{m}=\{\omega_{k}\}$ then
 \begin{equation*}
 A\left\|\psi \right\| ^{2}\leq \left\langle \sum \omega_{k}\left\langle \psi|\phi_{k} \right\rangle \phi_{k} |\psi \right\rangle\leq B\left\| \psi\right\| ^{2}
 \end{equation*}
 and
 \begin{equation*}
 A\left\|\psi \right\| ^{2}\leq \sum \omega_{k}
 \left|\left\langle \psi|\phi_{k} \right\rangle \right|^{2}\leq B\left\|\psi\right\| ^{2}.
 \end{equation*}
 So that there are constants $A>0$ and $B<\infty$ such that 
 \begin{equation}
 A\left\|\psi \right\| ^{2}\leq \sum \omega_{k}
 \left|\left\langle \psi|\phi_{k} \right\rangle \right|^{2}\leq B\left\| \psi\right\| ^{2}
 \end{equation}
 i.e. the pair $\{\omega_{k}\},\{\phi_{k}\}$ forms a $w-$ frame.\\
 On the other hand suppose that there exist constants $A>0$ and $B<\infty$ such that  $A\left\|\psi \right\| ^{2}\leq \sum \omega_{k}
 \left|\left\langle \psi|\phi_{k} \right\rangle \right|^{2}\leq B\left\| \psi\right\| ^{2}.$ It follows that $A\left\|\psi \right\| ^{2}\leq \left\langle \mathbf{M}_{\mathbf{m},\Phi}\psi|\psi \right\rangle\leq B\left\| \psi\right\| ^{2},$ for all $\psi\in V_{\mathbb{H}}^{L}. $ Hence $AI_{V_{\mathbb{H}}^{L}}\leq \mathbf{M}_{\mathbf{m},\Phi}\leq BI_{V_{\mathbb{H}}^{L}}.$ From Proposition \ref{pro1}, $\mathbf{M}_{\mathbf{m},\Phi}\in \mathcal{GL}^{+}(V_{\mathbb{H}}^{L}). $ Thereby $\mathbf{M}_{\mathbf{m},\Phi}$ is positive and invertible.\\
For iii$\Leftrightarrow$iv, suppose that there exist constants $A>0$ and $B<\infty$ such that 
\begin{equation}
A\left\|\psi \right\| ^{2}\leq \sum \omega_{k}
\left|\left\langle f|\phi_{k} \right\rangle \right|^{2}\leq B\left\| \psi\right\| ^{2},
\end{equation}
for all $\psi\in V_{\mathbb{H}}^{L}.$ Thereby
\begin{equation}
A\left\|\psi \right\| ^{2}\leq \sum 
\left|\left\langle \psi|\sqrt{\omega_{k}}\phi_{k} \right\rangle \right|^{2}\leq B\left\| \psi\right\| ^{2} ,
\end{equation} 
for all $\psi\in V_{\mathbb{H}}^{L},$ as  $\{\omega_{k}\}$ is the sequence of positive weight.  Hence $\{\sqrt{\omega_{k}}\psi_{k}\}$ is a frame in $V_{\mathbb{H}}^{L}.$ In similar argument the converse part can be obtained.\\
For i$\Leftrightarrow$iv, suppose that  $\{\phi_k\}$ is a frame and $\{\omega_{k}\}$ is a sequence of positive, semi-normalized weights. Then there exist constants $A>0$ and $B<\infty$ such that
\begin{equation}\label{411}
A\left\| \psi\right\| ^{2} \leq \sum \left |\left\langle \psi|\phi_{k}\right\rangle  \right|^{2}\leq B\left\|\psi\right\|^{2}, 
\end{equation}
for all $\psi\in V_{\mathbb{H}}^{L}$  
and there exist constants $b\geq a>0$ such that 
\begin{equation}\label{412}
a\leq\left| \omega_{k}\right| \leq b.
\end{equation}  
Let $\psi\in V_{\mathbb{H}}^{L},$ from equations \ref{411} and \ref{412}
	\begin{eqnarray*}
	\sum \left |\left\langle \psi|\sqrt{\omega_{k}}\phi_{k}\right\rangle  \right|^{2} &=& \sum\left| \omega_{k}\right|\left |\left\langle \psi|\phi_{k}\right\rangle  \right|^{2} \\
	&\leq&b B\left\|\psi\right\|^{2}. 
\end{eqnarray*}
In similar way one can prove that $\displaystyle \sum \left |\left\langle \psi|\sqrt{\omega_{k}}\phi_{k}\right\rangle  \right|^{2}\geq aA\left\| \psi\right\| ^{2}.$ 	Hence
\begin{equation*}
aA\left\| \psi\right\| ^{2}\leq \sum \left |\left\langle \psi|\sqrt{\omega_{k}}\phi_{k}\right\rangle  \right|^{2} \leq b B\left\|\psi\right\|^{2}, 
\end{equation*}
for all $\psi\in V_{\mathbb{H}}^{L}.$ Thereby $\{\sqrt{\omega_{k}}\psi_{k}\}$ is a frame in  $ V_{\mathbb{H}}^{L}$ with bounds $aA$ and $bB.$\\
On the other hand suppose that $\{\sqrt{\omega_{k}}\phi_{k}\}$ is a frame. Then there exist constants $A>0$ and $B<\infty$ such that
\begin{equation}\label{413}
A\left\| \psi\right\| ^{2} \leq \sum \left |\left\langle \psi|\sqrt{\omega_{k}}\phi_{k}\right\rangle  \right|^{2}\leq B\left\|\psi\right\|^{2}, 
\end{equation}
for all $\psi\in V_{\mathbb{H}}^{L}. $ It is clear that if $\{w_{k}\}$ is positive and semi-normalized then $\{w_{k}^{-1}\}$ is also positive and semi-normalized. Then there are bounds $s\geq r>0$ such that 
\begin{equation}\label{414}
r\leq\left| \omega_{k}^{-1}\right| \leq s.
\end{equation} 
Let $\psi\in V_{\mathbb{H}}^{L},$ from equations \ref{413} and \ref{414}
\begin{eqnarray*}
	\sum\left| \left\langle \psi|\phi_{k}\right\rangle \right|^{2}&=& \sum\left| \left\langle \psi|\sqrt{\omega_{k}w_{k}^{-1}}\phi_{k}\right\rangle \right|^{2}\\
	&=&\sum\left|\omega_{k}^{-1} \right| \left| \left\langle \psi|\sqrt{\omega_{k}}\phi_{k}\right\rangle \right|^{2}\\
	&\leq&sB\left| \psi\right\| ^{2}.
\end{eqnarray*}
Similarly one can get $\displaystyle\sum\left| \left\langle \psi|\phi_{k}\right\rangle \right|^{2}\geq rA\left| \psi\right\| ^{2}.$ It follows that 
\begin{equation}
rA\left| \psi\right\| ^{2}\leq\sum\left| \left\langle \psi|\phi_{k}\right\rangle \right|^{2}\leq sB\left| \psi\right\| ^{2},
\end{equation}
for all $\psi\in V_{\mathbb{H}}^{L}.$ Thereby $\{\psi_k\}$ is a frame in $ V_{\mathbb{H}}^{L} $  with frame bounds $rA \mbox{~and~} sB.$\\
For v $\Rightarrow$ i, suppose that $\{\omega_{k}\phi_{k}\}$ is a frame. Then there exists $A>0$ and $B<\infty$ such that
\begin{equation}\label{WF}
A\left\| \psi\right\| ^{2} \leq \sum \left |\left\langle \psi|\omega_{k}\phi_{k}\right\rangle  \right|^{2}\leq B\left\|\psi\right\|^{2}, 
\end{equation}
for all $\psi\in V_{\mathbb{H}}^{L}. $ 	We know that $\{\omega_{k}^{2}\}$ is positive and semi-normalized as $\{\omega_{k}\}$ is positive and semi-normalized. Thereby there exist constants $v\geq u>0$ such that
\begin{equation}\label{417}
u\leq\left| \omega_{k}^{2}\right| \leq v.
\end{equation}
Let $\psi\in V_{\mathbb{H}}^{L},$ from equations \ref{WF} and \ref{417},
\begin{eqnarray*}
	A\left\| \psi\right\| ^{2} &\leq& \sum \left |\left\langle \psi|\omega_{k}\phi_{k}\right\rangle  \right|^{2}\\&=&
	\sum \left| \omega_{k}\right|^{2} \left |\left\langle \psi|\phi_{k}\right\rangle  \right|^{2}\\
	&\leq&v\sum \left |\left\langle \psi|\phi_{k}\right\rangle  \right|^{2}.
\end{eqnarray*}
Hence $\displaystyle\frac{A}{v}\left\| \psi\right\| ^{2} \leq \sum \left |\left\langle \psi|\phi_{k}\right\rangle  \right|^{2}. $\\
	Similarly
\begin{eqnarray*}
B\left\|\psi\right\|^{2} &\geq&\sum \left |\left\langle \psi|\omega_{k}\phi_{k}\right\rangle  \right|^{2}\\
&=&	\sum \left| \omega_{k}\right|^{2} \left |\left\langle \psi|\phi_{k}\right\rangle  \right|^{2}\\
	&\geq&u\sum \left |\left\langle \psi|\phi_{k}\right\rangle  \right|^{2}.
\end{eqnarray*}
Hence $\displaystyle\sum\left |\left\langle \psi|\phi_{k}\right\rangle  \right|^{2}\leq\displaystyle\frac{B}{u}\left\| \psi\right\| ^{2}.$ 
It follows that $$\displaystyle\frac{A}{v}\left\| \psi\right\| ^{2} \leq \sum \left |\left\langle \psi|\phi_{k}\right\rangle  \right|^{2}\leq \displaystyle\frac{B}{u}\left\| \psi\right\| ^{2},  $$
for all $\psi\in V_{\mathbb{H}}^{L}. $
Therefore $\{\phi_k\}$ is a frame in $V_{\mathbb{H}}^{L}. $ \\
i$\Rightarrow$ v follows from Lemma \ref{WL1}. This completes the proof.
\end{proof}

\section{Conclusion}
We have shown that, in the discrete case, controlled frames, weighted frames and frame multipliers can be defined in a left quaternionic Hilbert space in almost the same way as they have been defined in a complex Hilbert space. The non-commutativity of quaternions does not cause any significant difficulty in the proofs of the results obtained. However, the complex numbers are two dimensional while the quaternions are four dimensional, hence the structure of quaternionic Hilbert spaces are significantly different from their complex counterparts, and therefore in the application point of view the theory developed in this note may provide some advantages or drawbacks.  The applications of the discrete case developed here and the corresponding continuous theory in the quaternionic setting are yet to be worked out.


\begin{thebibliography}{XXXX}
\bibitem{DU} Duffin, R.J., Schaeffer, A.C., \textit{A class of nonharmonic Fourier series}, Trans. Amer. Math. Soc.. {\bf 72} (1952) 341-366.
\bibitem{D1}Daubechies, I., Grossmann, A., Meyer, Y., \textit{Painless nonorthogonal expansions}, J. Math. Phys. {\bf 27} (1986) 1271-1283. 
\bibitem{D2}Daubechies, I., \textit{Ten lectures on wavelets}, SIAM. Philadelphia, 1992.
\bibitem{Ald} Aldroubi, A., Grochenig, K.H., \textit{Nonuniform sampling and reconstruction in shift invariance spaces}, SIAM Rev. {\bf 43} (2001) 585-620.
\bibitem{Han}Han, D., Larson, D.R., \textit{Frames, bases and group representations}, Memoris AMS. {\bf 697} (2000).
\bibitem{Lac} Lacey, M., Thiele,C., \textit{$L^{p}$ estimates on the bilinear Hilbert transform for $2<p<\infty$}, Annals of Math(2). {\bf 146} (1997) 693-724.
\bibitem{D3}Daubechies, I.,\textit{The wavelet transform, time-frequency localization and signal analysis}, IEEE Trans. Inform. Theory, {\bf 39} (1990) 961-1005.
\bibitem{Sh1}Shrohmer, T., Beaver,S., \textit{Optimal OFDM system design for time-frequency dispersive channels}, IEEE Trans. Communications, {\bf 51} (2003) 1111-1122. 
\bibitem{Sh2}Shrohmer, T.,   Heath Jr.,R.W., \textit{Grassmannian frames with applications to coding and communication}, Appl. Comput. Harmonic Anal., {\bf14} (2003) 257-275. 
\bibitem{Bod}Bodmann,B.,  Paulsen,V.I., \textit{Frames, graphs and erasures}, Linear Alg. and Appls.{\bf404} (2005) 118-146. 
\bibitem{GO1} Goyal, V.K., Kovacevic, J., Kelner, A.J.,
\textit{Quantized frame expansions with erasures}, Appl. Comp. Harm. Anal. {\bf 10} (2000), 203-233.

\bibitem{Bol}Bolckei,H., F. Hlawatsch,F.,  Feichtinger,H.G., \textit{Frame-theoretic analysis of oversampled filter banks}, IEEE Trans. Signal Process, {\bf46} (1998),  3256-3268. 

\bibitem{Ben}Benedetto,J.J., Powell, A.A., Yilmaz,O.,\textit{Sigma-Delta quantization and finite frames}, IEEE Trans. Inform. Theory, {\bf 52} (2006), 1990-2005. 
\bibitem{GO2}Goyal, V.K., Vetterli, M., Nguyen, T.T.,
\textit{Quantized overcomplete expansions in $\mathbb{R}^{n}:$ Analysis, Synthesis, and Algorithms}, IEEE Transactions on Information Thory. {\bf 10} (1998), 16-31.

 \bibitem{Can}E.J. Cand´es, E.J., Donoho,D.L.,\textit{ Continuous curvlet transform: II. Discretization and frames}, Appl. Comput. Harmonic Anal., {\bf19} (2005) 198-222.
\bibitem{Mar1} Margrave, G.F.,  Gibson, P.C., Grossman, J.P., Henley, D.C., Iliescu, V., Lamoureux, M.P.,  \textit{The Gabor transform, pseudiﬀerential operators, and seismic deconvolution}, Integrated Computer-Aided Engineering, {\bf12} (2005) 43-55. 
\bibitem{Eld}Eldar, Y.C., Forney,  G.D., \textit{Optimal tight frames and quantum measurement}, IEEE Trans. Inform. Theory, {\bf 48} (2002) 599-610.
\bibitem{Ad}Adler, S.L., {\em Quaternionic quantum mechanics and quantum fields}, Oxford University Press, New York, 1995.
\bibitem {KH} Khokulan, M., Thirulogasanthar, K., Srisatkunarajah, S., {\em Discrete frames on finite dimensional quaternion Hilbert spaces},  Axioms,{\bf 6, 3;}(2017) doi:10.3390/axioms6010003.
\bibitem{Kho} Khokulan,M., Thirulogasanthar,K., Muraleetharan,B.,
\textit{S-spectrum and associated continuous frames on quaternionic Hilbert spaces},  J.Geom.Phys.,{\bf 96} (2015) 107-122.	
\bibitem{Bal1} Balazs, P., \textit{Basic Definition and Properties of Bessel Multipliers}, Journal of Mathematical Analysis and Applications, {\bf325(1)} (2007) 71–585. 
\bibitem{Bal2} Balazs, P., Laback, B., Eckel, G., Deutsch, W.A., \textit{ Time-Frequency Sparsity by Removing Perceptually Irrelevant Components Using a Simple Model of Simultaneous Masking}, IEEE Transactions on Audio, Speech and Language Processing {\bf18(1)}(2010) 34-49. 
\bibitem{Bal3} Majdak, P., Balazs, P., Kreuzer, W., and Dorfler, M., \textit{A Time-Frequency Method for Increasing the Signal-To-Noise Ratio in System Identification with Exponential Sweeps}, Proceedings of the $36^{th}$ International Conference on Acoustics, Speech and Signal Processing, ICASSP 2011, Prag, 2011 
\bibitem{Wan} Wang, D., Brown, G.J., \textit{Computational Auditory Scene Analysis: Principles, Algorithms, and Applications}, Wiley-IEEE Press, 2006. 
\bibitem{Maj} Majdak, P., Balazs, P., Laback, B., \textit{Multiple Exponential Sweep Method for Fast Measurement of Head Related Transfer Functions}, Journal of the Audio Engineering Society {\bf55(7/8)} (2007) 623-637
\bibitem{Mar}Margrave, G.F., Gibson, P., Grossman, J.P., Henley, D.C., Iliescu, V.,  Lamoureux, M.P., \textit{The Gabor Transform, Pseudodiﬀerential Operators, and Seismic Deconvolution}, Integrated Computer-Aided Engineering {\bf12(1)}(2005) 43–55 
\bibitem{Bal4} Balazs, P., Antoine, J.P., Grybos, A., \textit{Weighted and Controlled frames}, International Journal of Wavelets, Multi resolution and Information Processing {\bf8(1)} (2010) 109-132. 	
\bibitem{Fa} Fashandi, M., {\em Some properties of bounded linear operators on quaternionic Hilbert spaces},  Kochi J. Math. {\bf 9} (2014) 127-135.	
\bibitem{Ric} Ghiloni, R.,  Moretti, V., Perotti, A., {\em Continuous slice functional calculus in quaternionic Hilbert spaces}, Rev. Math. Phys. {\bf 25} (2013), 1350006.
\bibitem{Bog} Bogdanova,I., Vandergheynst,P., Antoine,J.P., Jacques,L., Morvidone,M.,\textit{ Stereographic wavelet frames on the sphere}, Applied Comput. Harmon. Anal. {\bf19} (2005) 223-252.


	

\end{thebibliography}
\end{document}